\documentclass[letterpaper, 10 pt, conference]{ieeeconf}
\IEEEoverridecommandlockouts
\title{\LARGE \bf
Optimality of certainty equivalence in expected value problems for uncertain linear systems
}

\usepackage{amsmath,amsfonts,amssymb,graphicx,color,times,cite,subfigure,algorithm,algorithmic,color,psfrag,pstool}
\newcounter{remark}
\newenvironment{remark}{%
   \refstepcounter{remark}\begin{trivlist}\item[]\hspace*{.4cm}{\em
   Remark\,\theremark:\,}}
   {\cvd\end{trivlist}}

\newcommand{\cvd}{\hfill$\Box$\vspace{0em}\par\noindent}

\newtheorem{lemma}{Lemma}
\newtheorem{theorem}{Theorem}
\newtheorem{proposition}{Proposition}
\newtheorem{assumption}{Assumption}
\newtheorem{definition}{Definition}

\newcommand{\eps}{\varepsilon}

\newcommand{\beq}{\begin{equation}}
\newcommand{\eeq}{\end{equation}}
\newcommand{\beqar}{\begin{eqnarray}}
\newcommand{\eeqar}{\end{eqnarray}}
\newcommand{\beqarno}{\begin{eqnarray*}}
\newcommand{\eeqarno}{\end{eqnarray*}}
\newcommand{\ba}[1]{\begin{array}{#1}}
\newcommand{\ea}{\end{array}}

\DeclareMathAlphabet\mathbfcal{OMS}{cmsy}{b}{n}
\DeclareMathOperator{\Aut}{\mathrm{Aut}}

\usepackage{color}
\definecolor{red}{rgb}{0.738,0,0}

\author{Frank Chuang, Claus Danielson, and Francesco Borrelli
\thanks{F. Borrelli is Associate Professor in the 
				Dept. of Mechanical Engineering,
        University of California, Berkeley, USA 
        {\tt\small fborrelli@me.berkeley.edu}}
\thanks{F. Chuang is a graduate student in the Dept. of Mechanical Engineering, 
				University of California, Berkeley, USA 
				{\tt\small fc4321@berkeley.edu}
				}
\thanks{C. Danielson is a graduate student in the Dept. of Mechanical Engineering, 
				University of California, Berkeley, USA 
				{\tt\small claus.danielson@me.berkeley.edu}
				}
}
\begin{document}
\maketitle

% ======================================================================================
% Abstract
% ======================================================================================
\begin{abstract}
In this paper we study the optimality of the certainty equivalence approximation in robust finite-horizon optimization problems with expected cost.  We provide an algorithm for determining the subset of the state-space for which the certainty equivalence technique is optimal.  In the second part of the paper we show how patterns in the problem structure called symmetries can be used to reduce the computational complexity of the previous algorithm.  Finally we demonstrate our technique through numerical examples.  In particular we examine networked battery systems and radiant slab building control, for which we show the certainty equivalence controller is optimal over the entire operating range.  
\end{abstract}

% ======================================================================================
% Introduction
% ======================================================================================
\section{Introduction}
In this paper we consider finite-time expected value optimal control of linear systems with additive stochastic disturbance subject to robust constraints. We consider a cost separable in time so that dynamic programming can be applied.

In general finding the exact optimal feedback control law is intractable. However there are several approaches for approximating the optimal feedback law.  One approach is certainty equivalence in which the stochastic disturbance is replaced in the cost by its expected value.  The objective of this paper is determine the subset of states for which the certainty equivalence controller is optimal.  

In this paper we provide an algorithm for calculating a region of the state-space in which the certainty equivalence controller is optimal.  Our algorithm is based on dynamic programming.  At each time step a set of multiparametric programs is solved to obtain the cost to go.  In the second portion of the paper, we investigate how symmetry of the model predictive problem can be exploited to decrease computation time and memory usage of the explicit certainty equivalence controller.

In the numerical examples section, we apply our technique to a simple integrator system, battery network system, and building HVAC system. For these systems, we identify regions of the state-space for which certainty equivalence provides the optimal feedback solution. In the case of the radiant slab building control system, we find that certainty equivalence is valid for the entire of the operating regime. Where applicable, we also demonstrate the use of system symmetry to reduce computation time and memory usage.

For simple problems, such as the unconstrained linear quadratic control, the exact optimal solution can be computed via dynamic programming \cite{BET95}. For more complex problems tractable alternatives to computing exact feedback solutions to the expected-value problem are available, include using Monte Carlo simulations, affine disturbance feedback, open-loop input sequences, and certainty equivalence. For general distributions and costs, the problem is often solved approximately using Monte Carlo sampling \cite{SBLecture}. The effect of finite sampling with respect to the original expected value problem was investigated by Wang and Ahmed \cite{WangAhmed}. For certain distributions, such as Gaussian, affine feedback can be used to approximate the feedback solution and propogate the distribution forward. Goulart, Kerrigan, and Maciejowski \cite{goulart} detail the use of affine disturbance feedback in the robust control of linear systems with additive disturbance. The solution of the expected value problem using affine feedback subject to probabilistic constraints was addressed by Ma \cite{ma2012fast} in the context of chance-constrained stochastic MPC. While affine disturbance feedback is computationally efficient, it is conservative because, in general, the optimal feedback policies are non-linear. Bertsimas, Iancu, and Parrilo \cite{BIP} have proven the optimality of affine disturbance feedback for a specific class of 1-D problems. Meanwhile, Hadjiyiannis, Goulart, and Kuhn \cite{HGK} and more recently Van Parys, Goulart, and Morari \cite{VGM} have characterized the suboptimality of affine disturbance feedback in expected value problems. Alternatively, open-loop input sequences can be used which generally lead to even more conservative solutions. The advantage is the faster computation time over affine feedback. In the certainty equivalence principle, the random disturbance sequence is replaced by its expected value. The expected value is removed from the cost and a nominal optimization problem is solved instead. While based on potentially bad approximations, certainty equivalence often performs very well when applied to problems in economics \cite{SBLecture}. We examine the optimality of certainty equivalence and also how to use the approximation to compute explicit controllers.

Symmetry has been used extensively in numerous fields to reduce computational complexity. In recent years symmetry has been applied to optimization to solve linear-programs
\cite{Bodi2010}, semi-definite programs \cite{Gatermann2004}, and integer-programs \cite{Bodi2011}.  In \cite{Fagnani1991} and \cite{Hazewinkel1983} symmetry was studied in control theory to decompose large-scale systems into invariant subsystems.  In \cite{Cogill2008} the authors exploited symmetry to reduce the computational complexity of $\mathcal{H}_2$ and $\mathcal{H}_\infty$ controllers.  In \cite{Danielson2014} the authors studied symmetry in linear model predictive control.  This paper extends these results to dynamic programming to solve the expected value problem with robust constraints.

% --------------------------------------------------------------------------------------
% Problem Definition
% --------------------------------------------------------------------------------------
\section{Problem Definition}
Consider the linear time-invariant discrete-time system with additive disturbance
\begin{equation}
x_{t+1}=Ax_t+Bu_t+d_t,\: t\ge0
\label{eq:system}
\end{equation}
where $x_t\in\mathbb{R}^n$ is the system state, $u_t\in\mathbb{R}^p$ the controlled input, $d_t\in\mathbb{R}^n$ the disturbance, $A\in\mathbb{R}^{n\times n}$, and $B\in\mathbb{R}^{n\times p}$. The system is subject to the constraint
\begin{equation}
x_t\in \mathcal{X}_t\text{ and }
u_t\in \mathcal{U},~\forall t\ge0,
\label{eq: constraints}
\end{equation}
where $\mathcal{X}_t\subset\mathbb{R}^n$ and $\mathcal{U}\subset\mathbb{R}^p$ are polytopes. The disturbances $\{d_0,d_1,...\}$ are random variables which are independently distributed. We assume that
\[d_t\in \mathcal{D}_t,~\forall t\ge0,\]
where $\mathcal{D}_t\subset \mathbb{R}^n$ is a polytope. Note that the disturbances are not required to have zero mean. Thus our method can be extended to affine systems by simply lumping the affine term with the disturbance.

Consider the cost			
\begin{equation*}
E\left(f_N(x_N,u_N) + \sum_{t=0}^{N-1}f_t(x_t,u_t)\right),
\label{eq: evc}
\end{equation*}
where $N$ is a fixed horizon length and the functions $f_t: \mathbb{R}^n\times\mathbb{R}^p\to\mathbb{R}$ are jointly convex in $x_t$ and $u_t$ for all $0\le t\le N$. We are interested in finding the feedback control policies which minimize the above cost subject to constraints. That is, we are interested in the solution to the problem
\begin{equation}
\begin{aligned}
\min_{\pi_0,...,\pi_{N-1}}&E\left(f_N(x_N,u_N) + \sum_{t=0}^{N}f_t(x_t,u_t)\right)\\
\text{subject to }& x_t\in\mathcal{X}_t,\:\forall d_t\in\mathcal{D}_t\: \forall t\ge0
\label{eq: evp}
\end{aligned}
\end{equation}
where $u_t=\pi_t(x_t)$ and $\pi_t: \mathcal{X}_t \rightarrow \mathcal{U}$ is a mapping from the system state $x_t\in\mathbb{R}^n$ to the input space $u_t \in \mathcal{U}$ for $t=0,\dots,N-1$.

% --------------------------------------------------------------------------------------
% Exact Controller using Dynamic Programming
% --------------------------------------------------------------------------------------
\subsection{Exact Controller using Dynamic Programming}
Problem (\ref{eq: evp}) can be solved using dynamic programming in the following sense. The terminal cost is defined as
\[J^*_N(x_N)=f_N(x_N,u_N)\]
and for each time $t=N-1,\dots,0$ we calculate the cost-to-go by solving the following optimization problem 
\begin{equation}
\label{eq:exact-problem}
\begin{aligned}
J^*_{t}(x_t)=&\inf_{u_t\in\mathcal{U}}J_t(x_t,u_t)\\
\text{subject to }&  A x_t + B u_t  + d_t \in \mathcal{X}_{t+1}\:\forall d_t\in\mathcal{D}_t
\end{aligned}
\end{equation}
for $x_t \in \mathcal{X}_t$ where 
\begin{equation}
J_{t}(x_t,u_t)=f_t(x_t,u_t) + E_{d_t}\left(J^*_{t+1}(Ax_t+Bu_t+d_t)\right).
\label{eq: dprecurse}
\end{equation}
For each time $t$ the optimal control policy $u_t = \pi^*_t(x_t)$ is the optimizer $\pi^*_t: \mathcal{X}_t \rightarrow \mathcal{U}$ of Problem (\ref{eq:exact-problem}). Note that in general dynamic programming is intractable because $E_{d_t}\left(J^*_{t+1}(Ax_t+Bu_t+d_t)\right)$ often does not have a closed-form solution.

% --------------------------------------------------------------------------------------
% Certainty Equivalence Controller
% --------------------------------------------------------------------------------------
\subsection{Certainty Equivalence Controller}
One approach to obtain an approximation to the controller $\pi^*_t(x_t)$ is to use the certainty equivalence principle. The certainty equivalence controller can be obtained using dynamic programming as follows.  The terminal cost is defined as
\[\tilde{J}^*_N(x_N)=f_N(x_N,u_N)\]
and for each time $t = N-1,\dots,0$ we calculate the cost-to-go by solving the following optimization problem
\begin{equation}
\label{eq: cedp}
\begin{aligned}
\tilde{J}^*_{t}(x_t)=&\inf_{u_t\in\mathcal{U}} \tilde{J}_{t}(x_t,u_t)\\
\text{subject to } &A x_t + B u_t + d_t \in\mathcal{X}_{t+1}\:\forall d_t\in\mathcal{D}_t
\end{aligned}
\end{equation}
for $x_t \in \mathcal{X}_t$ where 
\[\tilde{J}_{t}(x_t)=f_t(x_t,u_t) + \tilde{J}^*_{t+1}(Ax_t+Bu_t+E(d_t)).\]  
For each time $t$ the optimal control policy $u_t = \tilde{\pi}^*_t(x_t)$ is the optimizer $\tilde{\pi}^*_t: \mathcal{X}_t \rightarrow \mathcal{U}$ of Problem (\ref{eq: cedp}).  

The difference between the exact and certainty equivalence control problems is the cost minimized at each stage.  The exact problem includes expected value of the cost-to-go $E(J^*_{t+1}(Ax+Bu+d))$.  The certainty equivalence problem includes the cost-to-go $\tilde{J}^*_{t+1}(Ax_t+Bu_t+E(d_t))$ with the expected disturbance. This renders the dynamic programming steps tractable for problems with quadratic cost and reasonable size.

% ======================================================================================
% Certainty Equivalence for Unconstrained Control
% ======================================================================================
\section{Certainty Equivalence for Unconstrained Control}
\label{sec: lqr}
It is well documented in the literature that the optimal unconstrained finite-horizon linear quadratic stochastic controller is equivalent to the finite-horizon LQR controller. For the remainder of the paper, we consider Problem (\ref{eq: evp}) with a fixed horizon $N$ and quadratic stage costs
\begin{subequations}
\label{eq: quadcost}
\begin{align}
f_N(x_N)&=x_N^TQ_Nx_N \\
f_t(x_t,u_t)&=x_t^TQ_tx_t+u_t^TR_tu_t
\end{align}
\end{subequations}
where $Q_t\succ 0$ and $R_t\succ 0$. Suppose for now that we do not impose constraints on the inputs ($\mathcal{U} = \mathbb{R}^p$) and states ($\mathcal{X}_t = \mathbb{R}^n$). Then if $E(d_t)=0$ for all $t$, the optimal controller is independent of the distribution of $d_t$ and is given recursively as
\[u_t^*=-(B^TP_{t+1}B+R_t)^{-1}B^TP_{t+1}A x_t,\]
where $P_{t+1}$ is the solution to the discrete Riccati equation
\begin{equation}
P_{t-1}=A^TP_{t}A-A^TP_{t}B(B^TP_{t}B+R_t)^{-1}B^TP_{t}A+Q_{t-1} 
\end{equation}
and $P_N=Q_N$. Therefore, the optimal controller is equivalent to the controller given by the certainty equivalence approach of replacing the disturbance $d_t$ by its expected value $E(d_t)=0$, which gives the conventional LQR controller. 

We briefly explain why the quadratic cost renders this result. Suppose the cost to go $J^*_{t+1}(x_{t+1})$ is a quadratic function so that $J^*_{t+1}(x_{t+1})=x^TQx+q^Tx+C$. Then by straightforward substitution we have
\begin{multline*}
J^*_{t+1}(Ax_t+Bu_t+d_t)\\
=g(x_t,u_t)+2x_t^TA^TQd_t+2u_t^TB^TQd_t+q^Td_t+d_t^TQd_t,
\end{multline*}
where $g(x_t,u_t)$ is a quadratic function of only $x_t$ and $u_t$.

Therefore, the cost-to-go is composed of terms linear in $d_t$ and a quadratic term which is a function of $d_t$ only. Thus when $E(d_t) = 0$, we have
\begin{align*}
&E_{d_t}(J^*_{t+1}(Ax_t+Bu_t+d_t))\\
&=g'(x_t,u_t)+2x_t^TA^TQE(d_t)+2u_t^TB^TQE(d_t)+q^TE(d_t), \\
&=g'(x_t,u_t) 
\end{align*}
where $g'(x_t,u_t)=g(x_t,u_t)+\mathrm{trace}(E(d_t^Td_t)Q)$. Therefore, $E_{d_t}(J^*_{t+1}(Ax_t+Bu_t+d_t))-J^*_{t+1}(Ax_t+Bu_t+E(d_t))=\mathrm{trace}(E(d_t^Td_t)Q)$, which is a constant. This means both cost functions have the same optimizer. The fact that the optimal control law is affine renders the cost-to-go at each time step to be quadratic. This implies that the certainty equivalence approximate of the problem will give the same optimizers as the original problem. 

% ======================================================================================
% Certainty Equivalence for Constrained Control
% ======================================================================================
\section{Certainty Equivalence for Constrained Control}

In this section we consider the optimality of the certainty equivalence controller for finite-time constrained optimal control problems.  

% --------------------------------------------------------------------------------------
% One-Step Certainty Equivalence Controller
% --------------------------------------------------------------------------------------
\subsection{One-Step Certainty Equivalence Controller}
From the discussion in the previous section, it is straightforward to show that for the linear quadratic stochastic optimal control problem with a one-step horizon $N=1$ the certainty equivalence controller is optimal.

% proposition
\begin{proposition}
Consider Problem (\ref{eq: evp}) with the horizon $N=1$ and cost (\ref{eq: quadcost}). Then the certainty equivalence controller is optimal ($\pi^*_0(x_0)=\tilde{\pi}^*_0(x_0)$) and the difference between the exact and certainty equivalence cost functions $J_0^*(x_0)-\tilde{J}_0^*(x_0)$ is a constant.
\label{prop: 1stepquad}
\end{proposition}
\begin{proof}
Note that it is sufficient to show that $E_{d_0}(J_1^*(Ax_0+Bu_0+d_0))-J_1^*(Ax_0+Bu_0+E(d_0))$ is a constant for all $x_0\in\mathbb{R}^n$, $u_0\in\mathcal{U}$, and $d_0\in\mathcal{D}_0$. The proof follows directly from the discussion in section \ref{sec: lqr}.
\end{proof}
In the following sections, we determine the subset of the state-space where this result can be extended for horizon lengths larger than $N=1$.  

% --------------------------------------------------------------------------------------
% Optimality of Certainty Equivalence
% --------------------------------------------------------------------------------------
\subsection{Optimality of Certainty Equivalence}
\label{sec: pert}
For certain subsets of the state space, it is possible to show that the certainty equivalence approximation will be the optimal solution to the original Problem (\ref{eq: evp}).  Before showing this result we review the relevant results for multiparametric quadratic programming \cite{BBM}.

% theorem
\begin{theorem}
Consider the following multiparametric program
\begin{subequations}
\label{eq:mpp}
\begin{alignat}{2}
J^*(x)=& \mathrm{minimize} ~&& J(z,x)  \\
&\mathrm{subject~to} ~&& z \in \mathcal{Z}(x) \\
& && x \in \mathcal{X}
\end{alignat}
\end{subequations}
where $z$ are the decision variables, $x$ are the parameters, $J(z,x) = \tfrac{1}{2} z^T Hz$ is a quadratic function, 
$\mathcal{F}(x) = \{ z : H z + G x \leq K \}$ is the feasible region, and $\mathcal{X}$ is a polytope.  Then
\begin{enumerate}
\item The optimizer $z^*: \mathcal{X} \rightarrow \mathcal{Z}$ is a continous piecewise affine on polyhedra function
\begin{align*}
z^*(x) = \begin{cases}
F_1 x + G_1 & \text{ for } x \in \mathcal{R}_1 \\ 
\qquad\vdots \\
F_r x + G_r & \text{ for } x \in \mathcal{R}_r
\end{cases} 
\end{align*}
\item The value function $J^*(x)$ is a convex piecewise quadratic on polyhedra function
\begin{align*}
J^*(x) = \begin{cases}
J_1^*(x) & \text{ for } x \in \mathcal{R}_1 \\ 
\quad\vdots \\
J_r^*(x) & \text{ for } x \in \mathcal{R}_r
\end{cases} 
\end{align*}
where $J_i^*(x)$ are quadratic functions.  
\item The closure of the critical regions $\mathcal{R}_i$ are a polyhedra.  The critical region partition is denoted by $\mathbfcal{R} = \{ \mathcal{R}_1, \dots, \mathcal{R}_r \}$.
\end{enumerate}
\label{thm: mpQP}
\end{theorem}
% --------------------------------------------------------------------------------------
% Determining where Certainty Equivalence is Optimal
% --------------------------------------------------------------------------------------
\subsection{Determining where Certainty Equivalence is Optimal}
\label{sec: recurseX}

Let $\mathbfcal{P}_{t+1}=\{\mathcal{R}^1_{t+1},...,\mathcal{R}^r_{t+1}\}$ be a P-collection of critical regions $\mathcal{R}_{t+1}^i \subseteq \mathcal{X}_{t+1}$ where the certainty equivalence controller is optimal.  In other words $\pi_{t+1}^*(x) = \tilde{\pi}_{t+1}^*(x)$ and $\tilde{J}^*_{t+1}(x)-J^*_{t+1}(x)$ is a constant for all $x \in \mathcal{R}_{t+1}^i$ and $\mathcal{R}_{t+1}^i \in \mathbfcal{P}_{t+1}$.  Note that at time $t=N-1$, the certainty equivalence array $\mathbfcal{P}_{N-1}$ is simply the set of critical regions of $J^*_{N-1}(x_{N-1})$ by Proposition \ref{prop: 1stepquad}.

For some critical region $\mathcal{R}_{t+1}^j \in \mathbfcal{P}_{t+1}$, consider the problem
\begin{equation}
\label{eq: evpsubset}
\begin{aligned}
J^*_{t}(x_{t})=&\inf_{u_t\in\mathcal{U}}J_t(x_t,u_t)\\
\text{subject to }&A x_t + B u_t + d_t \in\mathcal{R}^j_{t+1}\:\forall d_t\in\mathcal{D}_t
\end{aligned}
\end{equation}
and its certainty equivalence approximation
\begin{equation}
\label{eq: cesubset}
\begin{aligned}
\tilde{J}^*_{t}(x_{t})=&\inf_{u_t\in\mathcal{U}}\tilde{J}_t(x_t,u_t)\\
\text{subject to }&A x_t + B u_t + d_t \in\mathcal{R}^j_{t+1}\:\forall d_t\in\mathcal{D}_t.
\end{aligned}
\end{equation}

Since $J^*_{t+1}(x)$ is quadratic in $R^j_{t+1}$, we can use proposition \ref{prop: 1stepquad} to conclude that the optimizers of problems (\ref{eq: evpsubset}) and (\ref{eq: cesubset}) are equal. Thus for states $x_t \in \mathcal{X}_t$ where the optimizers of problem (\ref{eq: evpsubset}) are optimizers of problem (\ref{eq:exact-problem}), we know the certainty equivalence optimizers of problem (\ref{eq: cesubset}) are optimal.  The following proposition provides a condition for determining when the optimizers of problems (\ref{eq: evpsubset}) and (\ref{eq:exact-problem}) are equivalent.  

Before we state the proposition, we first define a few objects to be used. We assume $\mathcal{R}^j_{t+1}$ and $\mathcal{X}_{t+1}$ are normalized. Let $\mathcal{R}^j_{t+1}=\{x\in\mathbb{R}^n|R^xx\le R^c\}$ and $\mathcal{X}_{t+1}=\{x\in\mathbb{R}^n|P^x x\le P^c\}$ be the minimal representations of the two polytopes. Suppose $R^x$ and $P^x$ have $p$ and $q$ columns, respectively. Define $E^j_{t+1}=\{i\in\{1,...,p\}|R^x_i\neq P^x_j \:\forall j\in\{1,...,q\}\}$. That is, $E^j_{t+1}$ represents the row indices of constraints exclusive to $\mathcal{R}^j_{t+1}$ and not $\mathcal{X}_{t+1}$. Define $M^j_{t+1}=\{i\in\{1,...,q\}|P^x_i= R^x_j \text{ for some }j\in\{1,...,p\}\}$. 
% proposition
\begin{proposition}
\label{prop:constraint-test}
The problem (\ref{eq: evpsubset}) has the same optimizers as problem (\ref{eq:exact-problem}) on the critical regions for which the constraints indexed by $E^j_{t+1}$ are inactive.
\end{proposition}

% proof
\begin{proof}
We first show that the cost to go at each time step $J^*_{t}(x_t)$ is convex for every $t\in\{0,...,N\}$. We show this recursively. Observe that $J^*_{N}(x_N)$ is convex because it is a quadratic cost. Suppose at time $t$, $J^*_{t}(x_t)$ is convex. Since $x_t=Ax_{t-1}+Bu_{t-1}+d_{t-1}$ is an affine map from $(x_{t-1},u_{t-1})$ to $x_t$ for fixed $d_{t-1}$, the function $J^*_{t}(Ax_{t-1}+Bu_{t-1}+d_{t-1})$ is jointly convex in $(x_{t-1},u_{t-1})$ for fixed $d_{t-1}$. It was shown in \cite{CACL} that $E_{d_{t-1}}(J^*_{t}(Ax_{t-1}+Bu_{t-1}+d))$ is a convex function in $(x_{t-1},u_{t-1})$. Let $\mathcal{U}_t(x)=\{u\in\mathcal{U}:Ax+Bu+d\in\mathcal{X}_{t+1} \:\forall d\in\mathcal{D}_t\}$. Next, we show that $\mathcal{U}_t(x)$ is a convex point-to-set map. Let $x_1$ and $x_2$ be two initial states and $u_1\in \mathcal{U}_t(x_1)$ and $u_2\in\mathcal{U}_t(x_2)$. We must show that $\lambda u_1+(1-\lambda)u_2\in\mathcal{U}_t(\lambda x_1+(1-\lambda) x_2)$. For any $d\in\mathcal{D}_t$, we have
\begin{multline*}
A(\lambda x_1+(1-\lambda) x_2)+B(\lambda u_1+(1-\lambda)u_2)+d\\
=\lambda(Ax_1+Bu_1+d)+(1-\lambda)(Ax_2+Bu_2+d).
\end{multline*}
Since $\mathcal{X}_{t+1}$ is convex, the above equation shows that $\lambda x_1+(1-\lambda) x_2\in\mathcal{X}_{t+1}$, which implies $\lambda u_1+(1-\lambda)u_2\in\mathcal{U}_t(\lambda x_1+(1-\lambda) x_2)$. Using the convexity of $E_{d_{t-1}}(J^*_{t}(Ax_{t-1}+Bu_{t-1}+d))$ and $\mathcal{U}_{t-1}(x)$, the authors of \cite{FK86} showed that $J^*_{t-1}(x_{t-1})$ is convex.

Let $(\hat{u}_t^*,\hat{\lambda}_t^*)$ be primal and dual optimizers, respectively, for problem \ref{eq: evpsubset}. Let $(u_t,\lambda_t)$ be primal and dual variables, respectively, for problem \ref{eq:exact-problem}. Set $u_t=\tilde{u}_t^*$. Suppose the $i$th element of $\hat{\lambda}_t$ corresponds to the $i$th row of $R^x$ and the $i$th element of $\lambda_t$ corresponds to the $i$th row of $P^x$. For each $i\in M^j_{t+1}$, let the $i$th element in $\lambda_t$ be set equal to the $j$th element in $\hat{\lambda}_t^*$, where $j$ is such that $P^x_i= R^x_j$. For $i\in\{1,...,q\}\setminus M^j_{t+1}$, we set the $i$th element of $\lambda_t$ equal to $0$. Since $(\hat{u}_t^*,\hat{\lambda}_t^*)$ satisfies the KKT conditions for problem \ref{eq: evpsubset}, it follows that if the constraints indexed by $E^j_{t+1}$ are inactive, then $(u_t,\lambda_t)$ satisfy the KKT conditions for problem \ref{eq:exact-problem}. Since problem \ref{eq:exact-problem} is convex, we conclude that $\tilde{u}_t^*$ is an optimal solution for problem \ref{eq:exact-problem}.
\end{proof}

Using this proposition we can find the region where the certainty equivalence controller is optimal by solving the multi-parametric quadratic program (\ref{eq: cesubset}) and testing the active constraints for the resulting critical regions.  This is summarized in Algorithm \ref{ALG:CEequiv}.  

Algorithm \ref{ALG:CEequiv} maintains a P-collection $\mathbfcal{P}_t$ of critical regions where the certainty equivalence solution is optimal.  For each time $t=N-1,\dots,0$, a multi-parametric quadratic program is used to solve problem (\ref{eq: cesubset}) with a critical region $\mathcal{R}_{t+1}^j \in \mathbfcal{P}_{t+1}$.  This produces an array of critical regions $\mathbfcal{R} = \{\mathcal{R}_1,\dots,\mathcal{R}_r\}$.  For each critical region $\mathcal{R}_i \in \mathbfcal{R}$, Algorithm \ref{ALG:CEequiv} uses the constraint test from Proposition \ref{prop:constraint-test} to determine if certainty equivalence holds inside $\mathcal{R}_i$.  At termination this algorithm returns the P-collection $\mathbfcal{P}_0 \subseteq 2^{\mathcal{X}_0}$ where the certainty equivalence control is the optimal solution to (\ref{eq: evp}).

% algorithm
\begin{algorithm}
\caption{\textsc{Computing Certainty Equivalence State Subset}}
\label{ALG:CEequiv}

\begin{algorithmic}[1]
\vspace{1mm}
\STATE Solve (\ref{eq: cedp}) at time $t=N-1$.  Obtain the initial certainty equivalence P-collection $\mathbfcal{P}_{N-1} = \{\mathcal{R}^1_{N-1},...,\mathcal{R}^r_{N-1}\}$
\FOR{$t = N-1$ to $0$}
\FOR{each $\mathcal{R}^j_t \in \mathbfcal{P}_t$}
\STATE Solve (\ref{eq: cesubset}) at time $t-1$.  Obtain critical region array $\mathbfcal{R} = \{ \mathcal{R}_1,\dots,\mathcal{R}_j \}$
\FOR{each $\mathcal{R}_i \in \mathbfcal{R}$}
\IF{the constraints indexed by $E_t^{j}$ are inactive for region $\mathcal{R}_i$}
\STATE Add critical regions $\mathcal{R}_i$ to certainty equivalence partition $\mathbfcal{P}_{t-1}$
\ENDIF 
\ENDFOR
\ENDFOR
\ENDFOR
\vspace{1mm}
\end{algorithmic}
\end{algorithm}

% ======================================================================================
% Certainty Equivalence for Problems with Symmetries
% ======================================================================================
\section{Certainty Equivalence for Problems with Symmetries}

In this section we present a method for reducing the computational complexity of Algorithm \ref{ALG:CEequiv}.  Our modification exploits patterns in the problem structure called symmetries.  

% --------------------------------------------------------------------------------------
% Symmetry Groups
% --------------------------------------------------------------------------------------
\subsection{Definition of Symmetry}

In this section we define symmetry for Problem (\ref{eq: evp}) and show how symmetry affects the exact and certainty equivalence controllers.  

A symmetry of Problem (\ref{eq: evp}) is a state-space transformation $\Theta$ and input-space transformation $\Omega$ that preserves the dynamics, constraints, and stage cost.  

% definition
\begin{definition}
\label{def:prob-symmetry}
A linear symmetry of Problem (\ref{eq: evp}) is a pair of invertible matrices $(\Theta,\Omega)$ such that for $t = 0,\dots,N$ 
\begin{subequations}
\begin{align}
\Theta A = A \Theta \\
\Theta B = B \Omega
\end{align}
\end{subequations}
\begin{align}
f_t(\Theta x, \Omega u) = f_t(x,u)
\end{align}
\begin{subequations}
\begin{align}
\Theta \mathcal{X}_t = \mathcal{X}_t \\
\Omega \mathcal{U}_t = \mathcal{U}_t
\end{align}
\end{subequations}
and, $\mathbf{p}(d) = \mathbf{p}(\Theta d)$ for all $d_t \in \Theta \mathcal{D}_t = \mathcal{D}_t$ where $\mathbf{p}(d)$ is the probability density function for the disturbance $d_t$
\end{definition}
The set of all symmetries $(\Theta,\Omega)$ that satisfy Definition \ref{def:prob-symmetry} is a group denoted by $\Aut(MPC)$.  In \cite{Danielson2013b} a procedure was presented for identifying the symmetry group $\Aut(MPC)$ where each $f_t$ is quadratic and $\mathcal{X}_t$, $\mathcal{D}_t$, and $\mathcal{U}_t$ are polytopic sets for $t = 0,\dots,N$.  

Symmetries of Problem (\ref{eq: evp}) affect the exact and certainty equivalence controllers.  Proposition \ref{prop:symmetry-exact} shows that symmetries $(\Theta,\Omega) \in \Aut(MPC)$ relate the exact control law $\pi_t^*$ at different points in the state-space.  First we state and prove the following lemma.  

% lemma
\begin{lemma}
\label{lemma:mpp-symmetry}
Consider the multiparametric program (\ref{eq:mpp}) where the strictly convex cost $J(x,u)$ and feasible region $\mathcal{F}(x)$ satisfy
\begin{align*}
J(\Theta x,\Omega u) &= J(x,u) \\
\Omega \mathcal{F}(\Theta^{-1} x) &= \mathcal{F}(x) \\
\Theta \mathcal{X} &= \mathcal{X}.
\end{align*}
Then the optimal solution satisfies $\pi^*(x) = \Omega \pi^*(\Theta^{-1} x)$ for all $x \in \mathcal{X}$.  
\end{lemma}

% proof
\begin{proof}
First we show that $\Omega \pi^*(\Theta^{-1}x)$ is a feasible solution to the multiparametric program.  Note
\begin{align*}
\Omega \pi^*(\Theta^{-1}x) \in \Omega\mathcal{F}(\Theta^{-1}x) = \mathcal{F}(x)
\end{align*}
where $\Theta^{-1} x \in \Theta^{-1} \mathcal{X} = \mathcal{X}$.  Next we show $\Omega \pi^*(\Theta^{-1} x)$ is an optimal solution to the multiparametric program.  Suppose $\Omega \pi^*(\Theta^{-1} x)$ is suboptimal then
\begin{align*}
J(x,\pi^*(x)) < J(x,\Omega \pi^*( \Theta^{-1} x))
\end{align*}
which implies 
\begin{align*}
J(y,\Omega^{-1} \pi^*(\Theta y))< J(y,\pi^*(y))
\end{align*}
where $x = \Theta y \in \Theta \mathcal{X} = \mathcal{X}$.  However this contradicts the optimality of $\pi^*$ at $y \in \mathcal{X}$.  Thus $\pi^*(x)$ and $\Omega \pi^*(\Theta^{-1}x)$ are both optimal solutions of the multiparametric program (\ref{eq:mpp}).  Since the cost is strictly convex the solution of the multiparametric program (\ref{eq:mpp}) is unique.  Therefore $\pi^*(x)=\Omega \pi^*(\Theta^{-1}x)$ for all $x \in \mathcal{X}$.  
\end{proof}

% proposition
\begin{proposition}
\label{prop:symmetry-exact}
Let $\pi_t^*$ be the solution to (\ref{eq:exact-problem}).  Then for each $(\Theta,\Omega) \in \Aut(MPC)$ we have $\pi_t^*(\Theta x) = \Omega \pi_t^*(x)$ for all $x \in \mathcal{X}_t$.  
\end{proposition}

% proof
\begin{proof}
For each time $t$ the cost function $J_t(x,u)$ in (\ref{eq:exact-problem}) is strictly convex.  Therefore using Lemma \ref{lemma:mpp-symmetry} we can show $\pi_t^*(\Theta x) = \Omega \pi_t^*(x)$ for all $x \in \mathcal{X}_t$ if the cost to go $J_t(x,u)$ and the feasible region $\mathcal{F}_t(x) = \{ u : u \in \mathcal{U}, Ax + Bu + d \in \mathcal{X}_{t+1}, \forall d \in \mathcal{D} \}$ are symmetric.  

First we prove the feasible region $\mathcal{F}(x)$ of (\ref{eq:exact-problem}) is symmetric.  From Definition \ref{def:prob-symmetry} we have
\begin{align*}
\mathcal{F}_t(\Theta x) 
&= \{ u : u \in \mathcal{U}, A\Theta x + Bu + d \in \mathcal{X}_t, \forall d \in \mathcal{D} \} \\
&= \{ \Omega u' : u' \in \Omega \mathcal{U}, Ax + Bu' + d' \in \mathcal{X}_{t+1},\\
&\quad \forall d' \in \Theta \mathcal{D} \} \\
&= \Omega \mathcal{F}_t(x)
\end{align*}
where $u' = \Omega^{-1} u$ and $d' = \Theta^{-1} d$.  

Next we prove by induction that the cost $J_t(x,u)$ is symmetric.  This holds for $t=N$ by Definition \ref{def:prob-symmetry}.  For $t < N$ assume $J_{t+1}(x,u)$ and $\pi_{t+1}^*(x)$ are symmetric.  By Definition \ref{def:prob-symmetry} we have $f_t(x,u) = f_t(\Theta x,\Omega u)$ for each $(\Theta,\Omega) \in \Aut(MPC)$.  We need to show this holds for the second term $E(J^*_{t+1}(Ax+Bu+d))$ in the cost function (\ref{eq: dprecurse}).  By Definition  \ref{def:prob-symmetry} and the induction hypothesis
\footnotesize
\begin{align*}
&E(J^*_{t+1}(A\Theta x+B\Omega u+d)) \\
&= \int_{d\in \mathcal{D}_t} J_{t+1}\big( A\Theta x+B\Omega u+d, \pi^*_{t+1}(A\Theta x+B\Omega u+d) \big) d\mathbf{p}(d) \\
&= \int_{\Theta d' \in \mathcal{D}_t} J_{t+1}\big( \Theta( A x+B u+d'), \Omega \pi^*_{t+1}(Ax+Bu+d') \big) d\mathbf{p}(\Theta d') \\
&= \int_{d' \in \mathcal{D}_t} J_{t+1}\big( A x+B u+d', \pi^*_{t+1}(Ax+Bu+d') \big) d\mathbf{p}(d') \\
&=E(J^*_{t+1}(A x+B u+d)) 
\end{align*} 
\normalsize
where $d' = \Theta^{-1} d$.  Therefore by Lemma \ref{lemma:mpp-symmetry} we conclude $\pi_t^*(\Theta x) = \Omega \pi_t^*(x)$ for all $x \in \mathcal{X}_t$.  
\end{proof}

This proposition says that the feedback control law $\pi_t^*$ at points $x$ and $y = \Theta x$ is related by a linear transformation $\Omega$.  

For the certainty equivalence controller $\tilde{\pi}_t^*$ we have a strong result: in addition to relating the control law $\tilde{\pi}_t^*(x)$ at different points in the state-space, symmetries permute the critical regions $\mathbfcal{R}_t$ of the controller.  

% proposition
\begin{proposition}
\label{prop:symmetry-certain}
Let $\tilde{\pi}_t^*$ be the solution to (\ref{eq: cedp}).  Then for each $(\Theta,\Omega) \in \Aut(MPC)$ we have $\tilde{\pi}_t^*(\Theta x) = \Omega \tilde{\pi}_t^*(x)$ for all $x \in \mathcal{X}_t$.  Furthermore for any critical region $\mathcal{R}^i_t \in \mathbfcal{R}_t$ there exists $\mathcal{R}^j_t \in \mathbfcal{R}_t$ such that $\mathcal{R}^j_t = \Theta\mathcal{R}^i_t$.  
\end{proposition}

% proof
\begin{proof}
See \cite{Danielson2014}.  
\end{proof}

We say two critical regions $\mathcal{R}_i,\mathcal{R}_j \in \mathbfcal{R}$ are equivalent if there exists a state-space transformation $\Theta \in \mathcal{G} = \Aut(MPC)$ such that $\mathcal{R}_j = \Theta \mathcal{R}_i$.  The set of all critical regions equivalent to region $\mathcal{R}_i$ is called an orbit
\begin{align}
\mathcal{GR}_i = \{ \Theta \mathcal{R}_i : \Theta \in \mathcal{G} \} \subseteq \mathbfcal{R}.  
\end{align} 
The P-collection $\mathcal{GR}_i$ is the set of critical regions $\mathcal{R}_j = \Theta \mathcal{R}_i \in \mathbfcal{R}$ that are equivalent to critical region $\mathcal{R}_i$ under the state-space transformations $\Theta \in \mathcal{G} = \Aut(MPC)$.  

The set of critical region orbits is denoted by $\mathbfcal{R} / \mathcal{G} = \{ \mathcal{GR}_{1}, \dots, \mathcal{GR}_{r} \}$, read as $\mathbfcal{R}$ modulo $\mathcal{G}$, where $\{ \mathcal{R}_{1}, \dots, \mathcal{R}_{r}\}$ is a set that contains one representative critical region $\mathcal{R}_{j}$ from each orbit $\mathcal{GR}_{j}$.  With abuse of notation we will equate the set of critical region orbits $\mathbfcal{R} / \mathcal{G}$ with sets of representative critical regions $\mathbfcal{R} / \mathcal{G} = \{ \mathcal{R}_{1}, \dots, \mathcal{R}_{r}\}$.  

% --------------------------------------------------------------------------------------
% Symmetry Certainty Equivalence Algorithm
% --------------------------------------------------------------------------------------
\subsection{Symmetric Certainty Equivalence Algorithm}

In this section we use symmetry to reduce the computational complexity of Algorithm \ref{ALG:CEequiv}.  

The following theorem shows that if certainty equivalence holds on a critical region $\mathcal{R}_i \in \mathbfcal{R}$ then it holds on the orbit $\mathcal{GR}_i \subseteq \mathbfcal{R}$ of that region $\mathcal{R}_i$.   

% theorem
\begin{theorem}
If certainty equivalence holds on a critical region $\mathcal{R}_i \in \mathbfcal{R}$ then it holds for each critical region $\mathcal{R}_j = \Theta \mathcal{R}_i \in \mathcal{GR}_i$ in the orbit $\mathcal{GR}_i$.  
\end{theorem}

% proof
\begin{proof}
By definition of certainty equivalence on $\mathcal{R}$ we have 
\begin{align}
\pi_t^*(x) = \tilde{\pi}_t^*(x)
\end{align}
for all $x \in \mathcal{R}_i$.  By Propositions \ref{prop:symmetry-exact} and \ref{prop:symmetry-certain} we have $\Omega \pi_t^*(x) = \pi_t^*(\Theta x)$ and $\Omega \tilde{\pi}_t^*(x) = \tilde{\pi}_t^*(\Theta x)$.  Thus 
\begin{align}
\tilde{\pi}_t^*(\Theta^{-1} x) = \Omega^{-1} \tilde{\pi}_t^*(x) =  \Omega^{-1} \pi_t^*(x) =  \pi_t^*(\Theta^{-1} x)
\end{align}
for all $\Theta x \in \mathcal{R}_i$.  In other words $\tilde{\pi}_t^*(x) = \pi_t^*(x)$ for all $x \in \Theta \mathcal{R}_i$.  
\end{proof}

This theorem can be used to reduce the number of multi-parametric quadratic programs solved in Algorithm \ref{ALG:CEequiv}.  Algorithm \ref{algo:symmetry-certainty} is a modification of Algorithm \ref{ALG:CEequiv} that only tests one representative from each orbit for certainty equivalence.  

% algorithm
\begin{algorithm}
\caption{Compute Certainty Equivalence Region $\mathbfcal{P}_0 \subseteq \mathcal{X}_0$}
\label{algo:symmetry-certainty}
\begin{algorithmic}[1]
\vspace{1mm}
\STATE Solve (\ref{eq: cedp}) at time $t=N-1$.  Obtain P-collection $\mathbfcal{R}_{N-1}$ of critical regions.  Certain equivalence holds on $\mathbfcal{P}_{N-1} = \mathbfcal{R}_{N-1}$ for each set in this array by Prop \ref{prop: 1stepquad}.
\STATE Construct P-collection $\mathbfcal{P}_{N-1} / \mathcal{G}$ that contains one representative region $\mathcal{R}^i_{N-1}$ from each orbit $\mathcal{G} \mathcal{R}^i_{N-1}$ for $\mathcal{R}^i_{N-1} \in \mathbfcal{R}_{N-1}$.
\FOR{$t = N-1$ to $0$}
\FOR{each $\mathcal{R}^i_t \in \mathbfcal{P}_t / \mathcal{G}$}
\STATE Solve (\ref{eq: evpsubset}) at time $t-1$ with $\mathcal{R}^i_t$.  Obtain an array of critical regions $\mathbfcal{R} = \{ \mathcal{R}_1,\dots,\mathcal{R}_r \}$
\WHILE{critical region array $\mathbfcal{R}$ is not empty}
\STATE Select $\mathcal{R}_j \in \mathbfcal{R}$
\STATE Remove orbit $\mathcal{G}\mathcal{R}_j$ of $\mathcal{R}_j$ from $\mathcal{R}$
\IF{the constraints indexed by $E^i_t$ are inactive for region $\mathcal{R}_j$}
\STATE Add $\mathcal{R}_j$ to $\mathbfcal{P}_{t-1} / \mathcal{G}$
\ENDIF
\ENDWHILE
\ENDFOR
\ENDFOR
\STATE Construct full certainty equivalence P-collection $\mathbfcal{P}_0$ by calculating the orbit $\mathcal{G}\mathcal{R}^i_0$ of each element $\mathcal{R}^i_0$ of $\mathbfcal{P}_0 / \mathcal{G}$.
\vspace{1mm}
\end{algorithmic}
\end{algorithm}

Algorithm \ref{algo:symmetry-certainty} maintains a P-collection $\mathbfcal{P}_t / \mathcal{G}$ of representative regions $\mathcal{R}^i_t$ for each orbit $\mathcal{GR}^i_t$ where certainty equivalence holds.   This array is initialize by solving  the terminal multi-parametric quadratic program (\ref{eq: cedp}) to obtain the initial certainty equivalence P-collection $\mathbfcal{P}_{N-1} = \mathbfcal{R}$.   The representative certainty equivalence P-collection $\mathbfcal{P}_{N-1} / \mathcal{G}$ is constructed storing a single representative $\mathcal{R}^i_t \in \mathbfcal{P}_{N-1}$ from each of the orbits $\mathcal{GR}^i_t \subseteq \mathbfcal{P}_{N-1}$.

In the dynamic programming loop Algorithm \ref{algo:symmetry-certainty} solves multiparametric program (\ref{eq: evpsubset}) for each representative region $\mathcal{R}^i_t \in \mathbfcal{P}_t / \mathcal{G}$ of the certainty equivalence P-collection $\mathbfcal{P}_t$.  This produces an array $\mathbfcal{R}$ of critical regions.  For each orbit of critical regions $\mathcal{GR}_j \subseteq \mathbfcal{R}$, Algorithm \ref{algo:symmetry-certainty} test one representative $\mathcal{R}_j$ for certainty equivalence.  If certainty equivalence holds then the region is added to the representative array $\mathbfcal{P}_{t-1} / \mathcal{G}$.  Thus Algorithm \ref{algo:symmetry-certainty} only adds a single representative $\mathcal{R}_j$ from each critical region orbit $\mathcal{GR}_j$.  

Finally Algorithm \ref{algo:symmetry-certainty} uses symmetry to reconstruct the full certain equivalence P-collection $\mathbfcal{P}_0$ from the P-collection of representative regions $\mathbfcal{P}_0 / \mathcal{G}$.  

Algorithm \ref{algo:symmetry-certainty} requires solving $\sum_{t=0}^N | \mathbfcal{P}_t / \mathcal{G} |$ multiparametric quadratic programs verses the $\sum_{t=0}^N | \mathbfcal{P}_t |$ multiparametric programs solved in Algorithm \ref{ALG:CEequiv}.  Algorithm \ref{algo:symmetry-certainty} includes the additional task of calculating the orbit $\mathcal{GR}$ of polytopes $\mathcal{R}$.  However this can be accomplished efficiently using Algorithm \ref{algo:orbit}.  

% algorithm
\begin{algorithm}
\caption{Orbit $\mathcal{GR}_i$ of region $\mathcal{R}_i$ in P-collection $\mathbfcal{R}$}
\label{algo:orbit}
\begin{algorithmic}[1]
\vspace{1mm}
\STATE Find point $x \in \mathrm{int}(\mathcal{R}_i)$
\STATE Calculate orbit $\mathcal{G}x$ of $x$ under $\mathcal{G}$
\IF{$y \in \mathrm{int}(\mathcal{R}_j)$ for $\mathcal{R}_j \in \mathbfcal{R}$ and $y \in \mathcal{G}x$}
\STATE Add $\mathcal{R}_j$ to $\mathcal{GR}_i$
\ENDIF
\end{algorithmic}
\end{algorithm}

% ======================================================================================
% Disturbances with unbounded support
% ======================================================================================
\section{Disturbances with unbounded support}
\label{sec: unbnd}
The above discussion assumes that $\mathcal{D}_t$ is a compact set at all times $t$. However, many commonly used probability distributions, such as Gaussian, have unbounded support. Since the constraints can't be satisfied with absolute certainty, problem (\ref{eq: evp}) can be reformulated with probabilistic constraints as follows. 
\begin{equation}
\begin{aligned}
\min_{\pi_0,...,\pi_{N-1}}&E\left(\sum_{t=0}^{N-1}f_t(x_t,u_t)+f_N(x_N,u_N)\right)\\
\text{subject to }& P(x_t\in\mathcal{X}_t,u_t\in\mathcal{U})\ge 1-\eps\: \forall t\ge0
\label{eq: evprob}
\end{aligned}
\end{equation}
where $u_t=\pi_t(x_t)$, $\pi_t: \mathbb{R}^n\rightarrow \mathbb{R}^p$ is a mapping from the system state $x_t\in\mathbb{R}^n$ to the input space $u_t\in \mathbb{R}^p$ for $t=0,\dots,N-1$, and $0<\epsilon<1$.  

To apply the method discussed in section \ref{sec: pert}, we must have compact disturbance sets. The idea is to find a polytopic subset of the disturbance set $\tilde{\mathcal{D}}_t\subset \mathcal{D}_t$ such that $E(d_t|d_t\in\tilde{\mathcal{D}}_t)=E(d_t)$ and $P(d_t\in\tilde{\mathcal{D}}_t)=(1-\eps)^{\frac{1}{N}}$ and instead solve the robust problem
\begin{equation}
\begin{aligned}
\min_{\pi_0,...,\pi_{N-1}}&E\left(\sum_{t=0}^{N-1}f_t(x_t,u_t)+f_N(x_N,u_N)\right)\\
\text{subject to }& x_t\in\mathcal{X}_t,\:\forall d_t\in\tilde{\mathcal{D}}_t\: \forall t\ge0
\label{eq: evpconserv}
\end{aligned}
\end{equation}
where $u_t=\pi_t(x_t)$, $\pi_t: \mathcal{X}_t \rightarrow \mathcal{U}$ is a mapping from the system state $x_t \in\mathbb{R}^n$ to the input space $u_t \in \mathcal{U}$ for $t =0,\dots,N-1$. Note that this will result in a conservative solution of the probabilistic constraint problem. We make the following assumption about $d_t$.

% assumption
\begin{assumption}
\label{assump: symm}
$d_t$ has a probability density function $\mathbf{p}(x)$ such that $\mathbf{p}(E(x)+x)=\mathbf{p}(E(x)-x)$. 
\end{assumption}

Assumption \ref{assump: symm} says that the probability density function is symmetric about the mean.  Under this assumption, it is straightforward to show that if $\tilde{\mathcal{D}}_t-E(d_t)\subset\mathbb{R}^n$ is a symmetric Borel set, then $E(d_t|d_t\in\tilde{\mathcal{D}}_t)=E(d_t)$.

The idea is to guarantee with probability $P$ the exactness of certainty equivalence by constructing Borel sets $\tilde{\mathcal{D}}_t$ for each time step $t$ satisfying the following two assumptions.

% assumption
\begin{assumption}
\label{assump: borel}$\null$
\begin{enumerate}
\item $\tilde{\mathcal{D}}_t-E(d_t)\subset\mathbb{R}^n$ is symmetric.
\item $\mathbf{P}(d_t\in\tilde{\mathcal{D}}_t)=P^{\frac{1}{N}}$.
\end{enumerate}
\end{assumption}

For certain distributions, such as 1-D Gaussian, the sets $\tilde{\mathcal{D}}_t$ satisfying Assumptions \ref{assump: borel} are easily computed. For distributions where the computation of the set is not straightforward, generalized versions of the Chebyshev inequality can be employed. Olkin and Pratt \cite{OlkinPratt} provides the following bound for a random vector $(X_1,...,X_n)\in\mathbb{R}^n$. 
\begin{multline}
\label{eq: cheb}
P\left(\bigcap_{i=1}^n \frac{|X_i-\mu_i|}{\sigma_i}\le k_i\right)\\
\ge 1-\frac{\left(\sqrt{u}+\sqrt{n-1}\sqrt{n\sigma\frac{1}{k_i^2}-u}\right)^2}{n^2},
\end{multline}
where 
\[u=\sum_{i=1}^n \frac{1}{k_i^2}+2\sum_{i=1}^n\sum_{i<j}\frac{\rho_{ij}}{k_ik_j},\]
$\mu_i$ is the $i$th mean, $\sigma_i$ is the $i$th standard deviation, and $\rho_{ij}$ is the correlation between $X_i$ and $X_j$. Observe that the set $\bigcap_{i=1}^n \frac{|X_i-\mu_i|}{\sigma_i}\le k_i$ is a hypercube which is symmetric about the mean. Therefore, to satisfy assumption \ref{assump: borel}, we just need to solve for the $k_i$'s such that
\[1-\frac{\left(\sqrt{u}+\sqrt{n-1}\sqrt{n\sigma\frac{1}{k_i^2}-u}\right)^2}{n^2}\ge P^{\frac{1}{N}}\]

%-------------------------------------------------------------------------------------------------%
\section{Implementation}
The algorithm above returns a subset of the state space for which certainty equivalence is exact, which is the underlying set of the P-collection $\mathbfcal{P}_0$. In addition, the algorithm can also keep track of the optimal affine controllers in each critical region. The controller can then be implemented directly in a receding horizon controller.

The other alternative is to store only the P-collection of critical regions. Whenever the measured state $x_0\in\mathbfcal{P}_0$, one can solve the following problem to retrieve the optimal control.

\begin{equation}
\begin{aligned}
\min& f_N(\bar{x}_N,\bar{u}_N) + \sum_{t=0}^{N}f_t(\bar{x}_t,\bar{u}_t)\\
\text{subject to }& x_t\in\mathcal{X}_t\\
& u_t=K_tx_t+c_t,\:\forall d_t\in\mathcal{D}_t\: \forall t\ge0,
\label{eq: evpce}
\end{aligned}
\end{equation}

where $\bar{x}_{t+1}=A\bar{x}_t+Bu_t+E(d_t)$,$\bar{x}_0=x_0$, and $\bar{u}_t=K_t\bar{x}_t+c_t$. The authors in \cite{goulart} have detailed a method to solve the above problem using affine disturbance feedback, which transforms the problem into a tractable convex problem.

In the case that $x_0$ is not in $\mathbfcal{P}_0$, the common solution is to continue using the affine disturbance feedback controller. The authors in \cite{HGK} and \cite{VGM} detail the implementation of affine controllers in expected value problems and also methods to compute the suboptimality of such controllers. 
\begin{remark}
The entire methodology above can be trivially derived for linear objective cost. The only difference is that the piecewise quadratic cost are replaced by piecewise linear cost.
\end{remark}

% ======================================================================================
% Numerical Examples
% ======================================================================================
\section{Numerical Examples}

In this section we present three numerical examples that demonstrate our methodology.

% --------------------------------------------------------------------------------------
% Integrator System
% --------------------------------------------------------------------------------------
\subsection{Integrator System}

For our first example, consider a 2-D  discrete integrator system described by
\[x_{t+1}=x_{t}+u_t+d_t\]
where $x_t,u_t,d_t \in \mathbb{R}^2$.  Suppose for a horizon $N=3$ we would like to solve the problem \ref{eq: evprob} with cost
\[f_t(x_t,u_t)=x_t^Tx_t+u_t^Tu_t\text{ and }f_N(x_N)=x_N^Tx_N\]
and constraints $x_t\in[-10,10]$ and $u_t\in[-1,1]$ where $d_t\in[-0.5,0.5]$.

Using the method described section \ref{sec: pert}, we compute the set of states at each time step for which certainty equivalence is exact. At time step $0$, the set of states for which the certainty equivalence approximation is exact is plotted below in Figure \ref{fig: 2dce}

\begin{figure}[htbp]
\centering
\includegraphics[width=\linewidth]{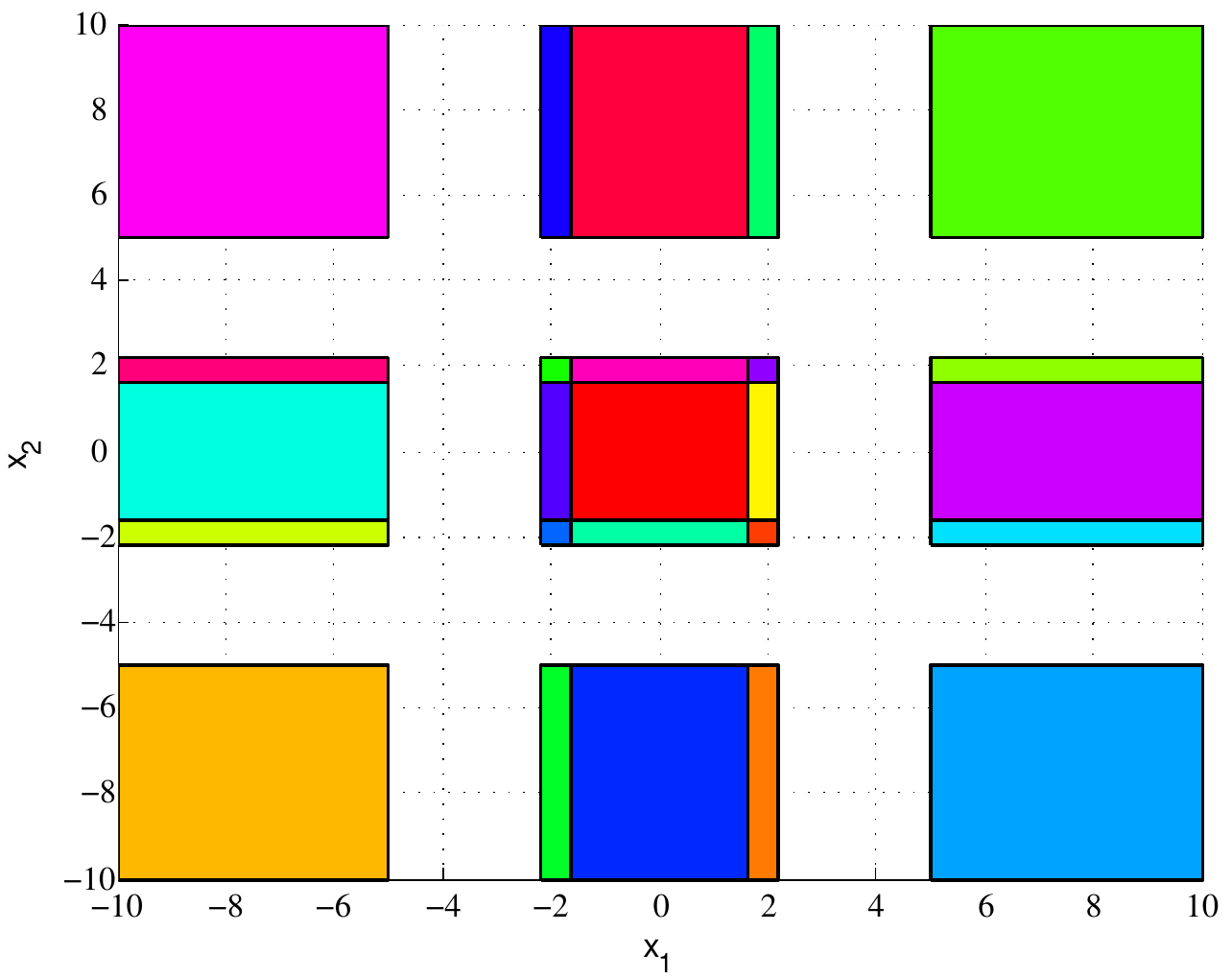}
\caption{Set $\mathbfcal{P}_0$ of initial states $x_0 \in \mathcal{X}_0$ for which the certainty equivalence controller $\tilde{\pi}_0^*$ is optimal $\pi_0^*(x_0) = \tilde{\pi}_0^*(x_0)$.}
\label{fig: 2dce}
\end{figure}
\subsubsection{Exploiting system symmetry}
The 2-D discrete integrator system presented above has symmetries which can be exploited to reduce computation time and memory usage. Since the matrix $A$, $B$, $Q$, and $R$ are identity, the symmetry group is determined by the constraints sets which are squares.  The symmetry group is the dihedral-4 group which consists of the four rotations by $90$ degrees and reflections about the horizontal, vertical, and both diagonal axis.  Using algorithm \ref{algo:symmetry-certainty}, we compute the following representative regions where certainty equivalence is exact.
\begin{figure}[htbp]
\centering
\includegraphics[width=\linewidth]{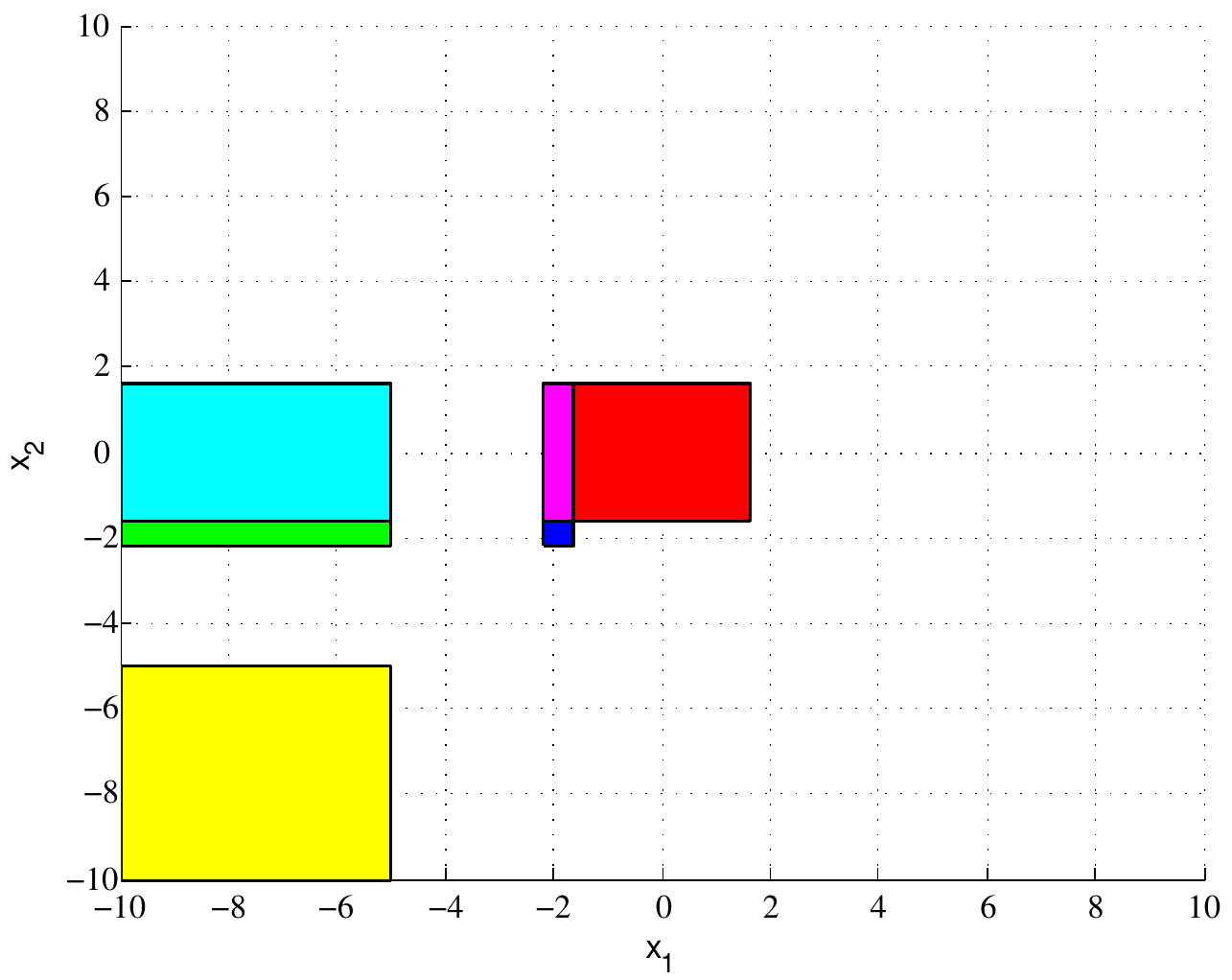}
\caption{Set of representative $\mathbfcal{P}_0/\mathcal{G}$ initial states $x_0 \in \mathcal{X}_0$ for which the certainty equivalence controller $\tilde{\pi}_0^*$ is optimal $\pi_0^*(x_0) = \tilde{\pi}_0^*(x_0)$.  The colored regions represent $\mathbfcal{P}_0/\mathcal{G}$ and gray regions represent $\mathbfcal{P}_0$.}
\label{fig: 2dcesym}
\end{figure}
To obtain the full set of states, we simply compute the orbit of each representative region. Note the decreased memory cost when storing just the representative regions. The biggest benefit, however, is the decreased computation time of solving 7 mpQP's (7.83 seconds) instead of 19 mpQP's (14.76 seconds).

% --------------------------------------------------------------------------------------
% Network battery system
% --------------------------------------------------------------------------------------

\subsection{Network Battery System}

In this example we consider a network of $n$ batteries connected in a ring as shown below.
\begin{figure}[htbp]
\centering
\includegraphics[scale=0.7]{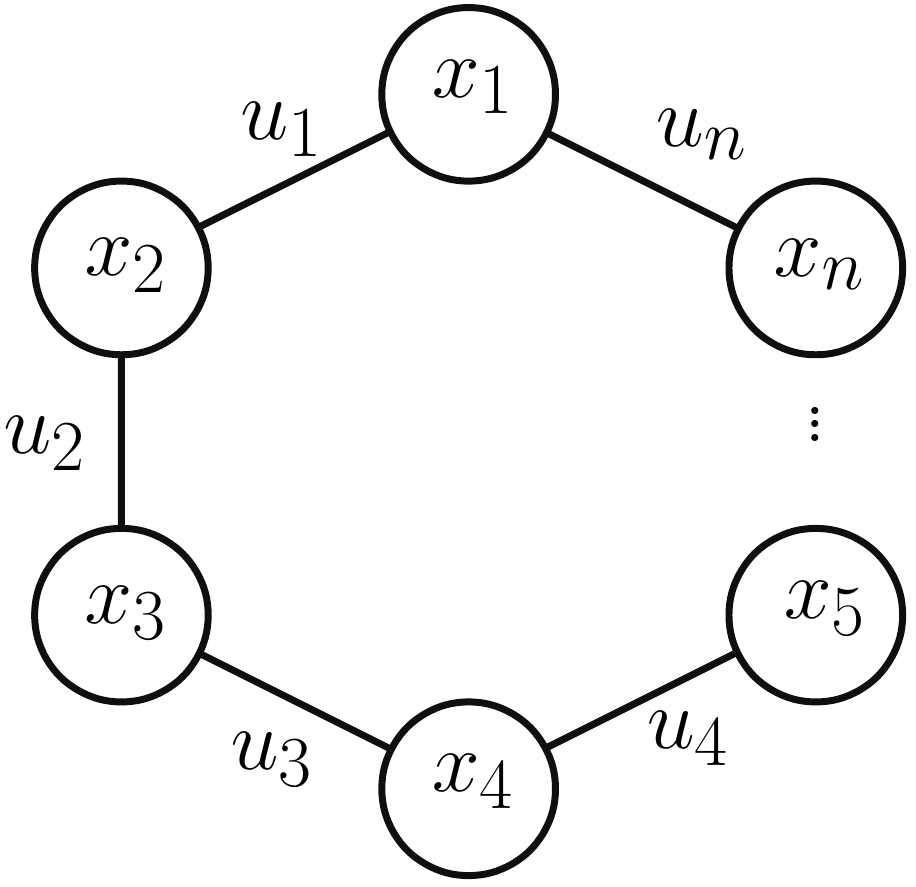}
\caption{Battery ring network}
\label{fig:batterynetwork}
\end{figure}
The states $x_t$ of the system are the amount of charge on each battery. The inputs $u_t$ are the current flows across each edge of the network. Suppose the maximum current is given by $I_{max}$ and the capacity of each battery is given by $C$. Then the system dynamics can be approximated by a linear system update equation of the form
\begin{equation}
x_{t+1}=x_{t}+\frac{I_{max}}{C}L(u_{t}+d_{t}),
\label{eq:batteryeqn}
\end{equation}
where $L$ is the Laplacian matrix of the graph and $d_t$ is a stochastic disturbance to the edge flows. The constraints on the system are $x_t\in[0,1]$, $u_t\in[-1,1]$, and $d_t\in[-0.1,0.1]$. We are interested in balancing the charges on the battery while minimizing the amount of charge moved on each edge. The problem (\ref{eq: evp}) with cost (\ref{eq: quadcost}) can be directly applied to solve this problem with
\[Q=I_{n}-\frac{1}{n}J_n \text{ and }R=10^{-6}I_n,\] 
where $J_n$ is a $n\times n$ matrix of ones. The set of symmetries of this problem is the dihedral group $D_n$, which can be exploited to reduce computation time and storage requirements.

We solved the problem with $n=5$, $I_{max}=5$,$C=3.6\cdot 10^5$, and $N=2$. The table below compares the solution times and number of critical regions with and without the use of symmetry.

\begin{table}
\center
\begin{tabular}{|c|c|c|}
\hline
&Computation time (s)&Number of critical regions\\
\hline
N=1&10.8&211\\
\hline
N=2&6,580&1998\\
\hline
N=3&63,800&8684\\
\hline
\end{tabular}
\caption{Battery network without symmetry}
\end{table}

\begin{table}
\center
\begin{tabular}{|c|c|c|}
\hline
&Computation time (s)&Number of critical regions\\
\hline
N=1&14.1&26\\
\hline
N=2&647&213\\
\hline
N=3&16,000&904\\
\hline
\end{tabular}
\caption{Battery network with symmetry}
\end{table}
% --------------------------------------------------------------------------------------
% Radiant Slab System
% --------------------------------------------------------------------------------------
\subsection{Radiant Slab System}

Consider the following radiant-slab system implemented at the Brower Center in Berkeley, CA. The system can be represented by the state vector $T_t=\begin{bmatrix} T_{slab,t}&T_{room,t}\end{bmatrix}^T$, where $T_{slab}$ is the temperature of the radiant slab and $T_{room}$ is the temperature of the room. Let $u_t$ be the temperature of the water supplied to the radiant slab. The radiant slab system can be approximated by a linear system update equation of the form
\begin{equation}
T_{t+1}=AT_{t}+Bu_{t}+Wd_{t},
\label{eq:browereqn}
\end{equation}
where
\[A=\begin{bmatrix} 0.9579&0.0406\\0.0093&0.9883\end{bmatrix}, B=\begin{bmatrix} 0.0016\\0\end{bmatrix},W=\begin{bmatrix}0\\0.0025\end{bmatrix},\]
and $d_{t}$ is the outside air temperature at time $t$, with time measured in hours. The parameters in equation (\ref{eq:browereqn}) were identified by performing step-tests on the actual building.

We are interested in controlling the water temperature supplied to the slabs to maintain the room air temperature close to a comfortable temperature of $70^\circ F$. The supply water temperature is constrained to be within $55^\circ F$ and $90^\circ F$. We investigate controlling the building temperature on a hot summer day, with a $48$-hour outside air temperature prediction, $OAT_t$, as shown in Figure (\ref{fig:oat}).

\begin{figure}[htbp]
\centering
\includegraphics[scale=0.5]{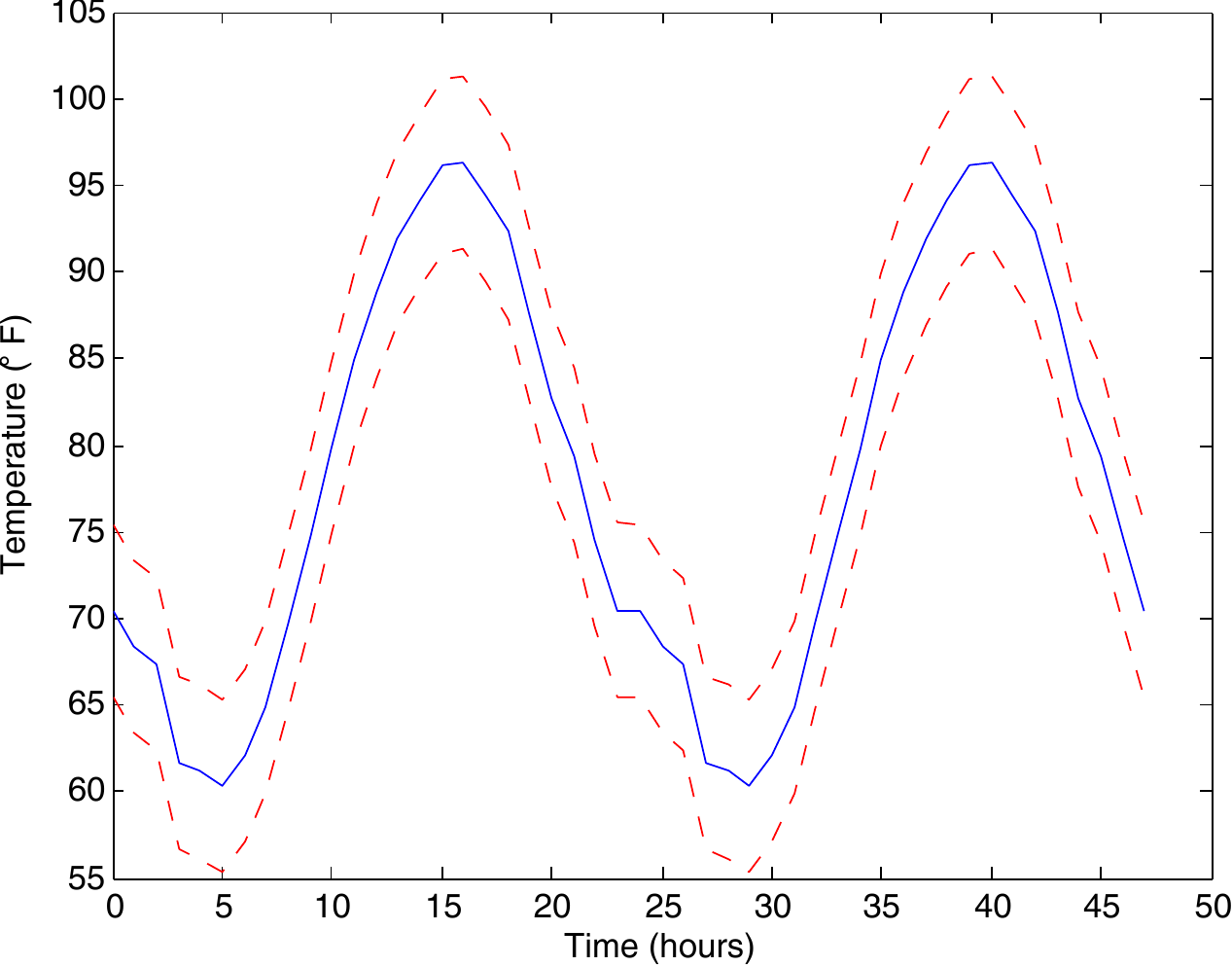}
\caption{Outside air temperature}
\label{fig:oat}
\end{figure}
We assume that the weather prediction has a $5$ degree radius uncertainty, which is shown by the dotted bounding lines above and below the nominal temperature profile. Suppose that we wish to maintain the room temperature, $T_{r,t}$, close to an optimal temperature of $70^\circ F$ while minimizing energy usage. Suppose that the water supply temperature, $u_t$, has a nominal temperature of $70^\circ F$ and that changing the water temperature from the nominal temperature will require energy. Suppose for a horizon of $N$ hours we would like to minimize the cost
\[E\left(\sum_{t=0}^{N-1} \left[(T_{r,t}-70)^2+\rho (u_t-70)^2\right]+(T_{r,N}-70)^2\right)\]
subject to the robust constraint $\begin{bmatrix}55\\60\end{bmatrix}\le T_t\le \begin{bmatrix}90\\80\end{bmatrix}$ and $u_t\in[55,90]$.
In order to write the cost in the form \ref{eq: quadcost}, we introduce new states $\tilde{T}_t=T_t-\begin{bmatrix}0\\70\end{bmatrix}$ and $\tilde{u}_t=u_t-70$. By straightforward substitution, the state update equation becomes
\[\tilde{T}_{t+1}=A\tilde{T}+B\tilde{u}+A\begin{bmatrix}0\\70\end{bmatrix}+70B-\begin{bmatrix}0\\70\end{bmatrix}+Wd_t\]
and the cost becomes
\[f_t(\tilde{T}_t,\tilde{u}_t)=\tilde{T}^T_t\begin{bmatrix}0&0\\0&1\end{bmatrix}\tilde{T}_t+\rho \tilde{u}_t^2,f_N(\tilde{T}_N)=\tilde{T}^T_N\begin{bmatrix}0&0\\0&1\end{bmatrix}\tilde{T}_N\]
For a horizon of $24$ hours, the figure \ref{fig: brower} shows the set of initial states $T_0$ such that the certainy equivalence approximation can be used to obtain an exact solution to problem \ref{eq: evp}.

\begin{figure}[htbp]
\centering
\includegraphics[width=\linewidth]{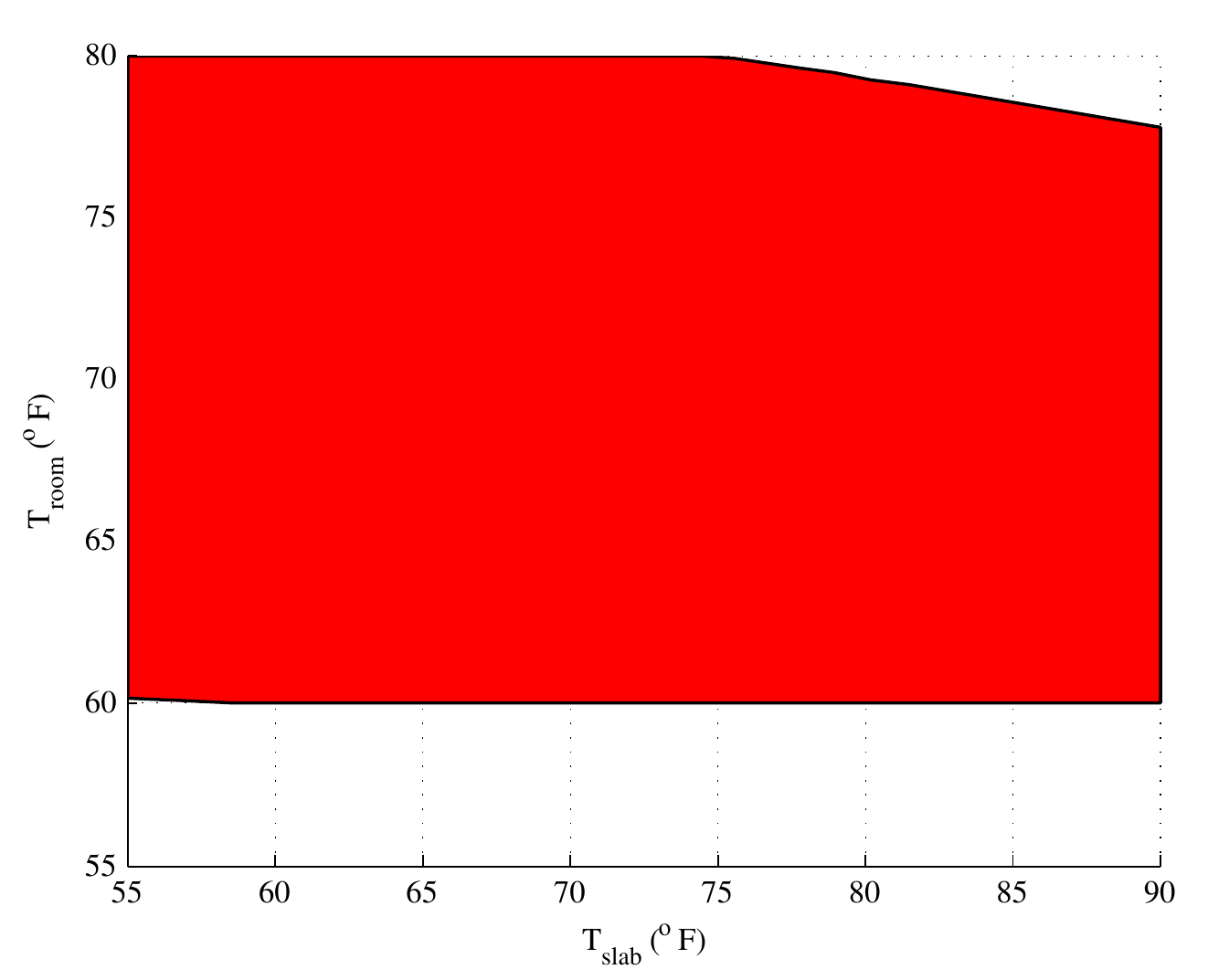}
\label{fig: brower}
\caption{$T_0$ for which CE is exact}
\end{figure}

The plot shows that for our radiant-slab system, the set of states for which certainty equivalence can be applied covers almost the entire operating regime. This shows that for the system and problem under consideration, there is little value in knowing the distribution of the disturbance beyond the first moment.

% ======================================================================================
% Conclusions
% ======================================================================================
\section{Conclusion}
This paper considered finite-time expected value optimization problems for linear systems with additive stochastic disturbance subject to robust constraints. We considered problems with quadratic cost separable in time so that dynamic programming can be applied. We presented an algorithm to compute regions of the state space such that the solution over feedback policies that satisfies robust constraints and minimizes the expected cost is the solution obtained by certainty equivalence. We also presented an algorithm which takes advantage of symmetries in the MPC problem to drastically reduce computation time and memory requirements. 
The algorithm was demonstrated on three numerical problems including a model of the Brower Center in Berkeley, CA. We showed that for the radiant-slab system, the certainty equivalence approximation is exact for a large portion of the operating regime.  The methodology of this paper allowed us to rigorously confirm our intuition that this system with its high capacitance should be resistant to variations in the disturbance. We also demonstrated with the integrator and battery network systems that symmetries can drastically reduce computation time and memory requirements.

% ======================================================================================
% Bibliography
% ======================================================================================
\bibliography{CEBound}
\bibliographystyle{IEEEtran}
\end{document}